\documentclass[11pt,reqno]{amsart}
\usepackage[utf8]{inputenc}
\usepackage{amssymb}
\usepackage{amsmath}
\usepackage{amsfonts}
\usepackage{mathtools}
\usepackage{hyperref}
\usepackage{theoremref}
\usepackage{color}
\usepackage{tikz}
\usetikzlibrary{matrix,arrows,cd}
\usepackage{xcolor}
\input diagxy
\usepackage{comment}
\usepackage[normalem]{ulem} 
\usepackage{setspace}
\usepackage{faktor}

\newtheorem{theorem}{Theorem}[section]
\theoremstyle{plain}

\newtheorem{corollary}[theorem]{Corollary}
\newtheorem{lemma}[theorem]{Lemma}
\newtheorem{proposition}[theorem]{Proposition}
\numberwithin{equation}{section}
\theoremstyle{definition}

\newtheorem{remark}[theorem]{Remark}
\newtheorem{definition}[theorem]{Definition}
\newtheorem{example}[theorem]{Example}

\newcounter{tempcomment}[section]

\newtheorem{problem-comment-thm}[tempcomment]{Problem}
\newtheorem{comment-thm}[tempcomment]{Comment}

\newcommand{\FACT}{\mathsf{FAct}}

\newcommand{\ID}{\mathtt{id}}

\newcommand{\END}{\mathtt{End}}

\newcommand{\ot}{\otimes}
\newcommand{\im}{\mathrm{im}}
\newcommand{\sk}[1]{\left\langle\,#1\,\right\rangle}

\title
{On connections between Morita semigroups and strong Morita equivalence}
\author{Alvin Lepik}
\thanks{This work was partially supported by the Estonian Research Council grant PRG1204.}
\begin{document}
\address{Institute of Mathematics and Statistics, University of Tartu, Narva rd 18, Tartu, Estonia}
\email{alvin.lepik@ut.ee}
\begin{abstract}
    A surjective Morita context connecting semigroups $S$ and $T$ yields a Morita semigroup and a strict local isomorphism from it onto $S$ along which idempotents lift. We describe strong Morita equivalence of firm semigroups in terms of Morita semigroups and isomorphisms. We also generalize some of Hotzel's theorems to semigroups with weak local units. In particular, the Morita semigroup induced by a dual pair $\beta$ over a semigroup with weak local units can be identified with $\Sigma ^\beta$.
\end{abstract}
\maketitle

\section{Introduction}
The study of Morita equivalence began in the theory of rings with identity. In the seventies, Banaschewski \cite{B72} and Knauer \cite{K72} independently developed Morita theory of monoids regarding monoids to be Morita equivalent if the categories of right acts, satisfying the identity $x1=x$, are equivalent. Banaschewski \cite{B72} also showed that categories of right acts over semigroups are equivalent only if the semigroups themselves are isomorphic. Further developments were brought about in the nineties, when Talwar \cite{T} provided a non category theoretical notion of Morita equivalence for semigroups with local units and extended it to the subclass of factorizable semigroups. Since then, Morita equivalence of semigroups with local units has been described by Lawson \cite{Lw} and Laan and M\'{a}rki \cite{LM} with an array of qualitatively different conditions.
\vspace{1em}

In the present paper, we mainly focus on strong Morita equivalence of semigroups and seek to generalize results from Hotzel \cite{Hz}  and Laan and M\'{a}rki \cite{LM}, which is the subject of Sections 2 and 3. In particular, we study relationships between Morita semigroups, Rees matrix semigroups and certain semigroups of adjoint pairs of endomorphisms of acts. Among other things we prove in Section 2 that two firm semigroups are strongly Morita equivalent if and only if any of them is isomorphic to a surjectively defined Morita semigroup over the other. In Section 3 we turn to dual pairs of acts, introduced by Hotzel \cite{Hz}, which he uses to describe completely $0$-simple semigroups. Hotzel's work has also inspired a number of ring theoretic papers. \'{A}nh and M\'{a}rki \cite{AM} describe rings with minimal one-sided ideals in terms of Rees matrix rings and dual pairs of modules. \'{A}nh \cite{Anh} describes Morita equivalence of rings with local units in terms of locally projective pairs (these are dual pairs satisfying further restrictions) and tensor product rings. We use dual pairs to deduce a sufficient condition for strong Morita equivalence of semigroups with weak local units.
\vspace{1em}

Throughout this paper, $S$ denotes a semigroup. We are considering the following subclasses of semigroups, listed in ascending order containment-wise, where all containments are proper.
\begin{definition}
A semigroup $S$
\begin{enumerate}
    \item[i)] has \textbf{local units} if for every $s\in S$, there exist idempotents $e,f\in S$ such that $fs=s=se$;
    \item[ii)] has \textbf{weak local units} if for every $s\in S$, there exist $u,v\in S$ such that $us=s=sv$;
    \item[iii)] is \textbf{firm} if the map
    \[
        S\ot _SS \to S,\quad s\ot s' \mapsto ss'
    \]
    is bijective;
    \item[iv)] is \textbf{factorizable} if for every $s\in S$, there exist $s_1,s_2\in S$ such that $s=s_1s_2$.
\end{enumerate}
\end{definition}
For subsets $U,V\subseteq S$ we write $UV := \{uv \mid u\in U, v\in V\}$. For singleton subsets we write $aS := \{a\}S$. A semigroup act $A_S$ (${}_SA$) is called \textbf{unitary} if $AS=A$ ($SA=A$). An $(S,T)$-biact ${}_SA_T$ is \textbf{unitary} if $A$ is unitary both as a left $S$-act and a right $T$-act. Semigroup theoretic notions that are not explicitly defined in this paper are covered in Howie's book \cite{H}. Subject matter pertaining to tensor products of acts is covered in \cite{KKM}.
\vspace{1em}

A central notion is that of a Morita context due to Talwar \cite{T}.
\begin{definition}
A \textbf{Morita context} connecting semigroups $S$ and $T$ is a six-tuple \\$(S,T,{}_SP_T,{}_TQ_S,\theta,\phi)$, where ${}_SP_T$ is an $(S,T)-$biact, ${}_TQ_S$ is a $(T,S)-$biact and
\begin{align*}
    \theta : {}_S(P \ot _T Q)_S \to {}_SS_S\quad\mbox{and}\quad
    \phi : {}_T(Q\ot _SP)_T \to {}_TT_T
\end{align*}
are biact morphisms satisfying the identities
\begin{align*}
     \theta (p\ot q)p' = p\phi (q\ot p')\quad \mbox{and}\quad q'\theta (p\ot q) = \phi (q'\ot p)q.
\end{align*}
A Morita context is called
\begin{enumerate}
    \item[i)] \textbf{unitary} if the biacts are unitary;
    \item[ii)] \textbf{surjective} if the biact morphisms are surjective.
\end{enumerate}
Semigroups that are connected by a unitary surjective Morita context are called \textbf{strongly Morita equivalent} \cite{T}. 

\end{definition}
Strong Morita equivalence is an equivalence relation on the subclass of factorizable semigroups. In fact, strong Morita equivalence can only occur between factorizable semigroups by Proposition 1 in \cite{LM}. From the category theoretical perspective, semigroups $S$ and $T$ are called \textbf{Morita equivalent} if the categories of firm acts $\FACT _S$ and $\FACT _T$ are equivalent (cf. \cite{Lw}). However, in the subclass of factorizable semigroups, it is sufficient to consider strong Morita equivalence, which coincides with the category theoretical Morita equivalence by Theorem 4.11 in \cite{LR}.

\section{Morita semigroups and strong Morita equivalence}
The following definition is due to Talwar \cite{T}.
\begin{definition}
Let ${}_SP$ and $Q_S$ be $S$-acts. A \textbf{Morita semigroup} over $S$ defined by $\sk{,}$ is the set $Q\otimes _SP$ with multiplication
\[
    (q\otimes p)(q'\otimes p') := q\otimes \sk{p,q'}p',
\]
where $\sk{,}:{}_SP\times Q_S \to {}_SS_S$ is an $(S,S)$-biact morphism. The Morita semigroup is
\begin{enumerate}
    \item[i)] \textbf{unitary} if $P$ and $Q$ are unitary $S$-acts;
    \item[ii)] \textbf{surjectively defined} if the map $\sk{,}$ is surjective.
\end{enumerate}
\end{definition}
\begin{example}\label{ex1.3}
Every Morita context gives, in a natural way, rise to two 
Morita semigroups. Let arbitrary semigroups $S$ and $T$ be connected by a Morita context \\
$(S,T,{}_SP_T,{}_TQ_S,\theta,\phi)$. Then 
using the biact morphism
\[
\sk{,}:{}_SP\times Q_S \to {}_SS_S, \;\; 
(p,q)\mapsto \theta(p\otimes q),
\]
we can turn $Q\otimes _SP$ into a Morita semigroup with multiplication 
\[
(q\ot p)(q'\ot p') = q\ot \theta(p\ot q')p' 
= q\ot p\phi(q'\ot p').
\]

The equalities
\[
    \phi ((q\ot p)(q'\ot p')) = \phi (q\ot p\phi (q'\ot p')) = \phi (q\ot p)\phi (q'\ot p')
\]
yield that $\phi$ is a semigroup morphism. In a similar way, $P\otimes_T Q$ is a Morita semigroup.
\end{example}
Given a Morita context with morphisms $\theta$ and $\phi$, these, of course, need not be isomorphisms, but they do have good properties, in general.
\begin{definition}
We say a semigroup morphism $\varphi :S\to T$ is \textbf{almost injective} if it is injective on all subsemigroups of the form $aSb$, where $a\in Sa$ and $b\in bS$. An almost injective semigroup morphism is called a \textbf{strict local isomorphism} if it is also surjective. \textbf{Idempotents lift along} $\varphi$, if for every $f\in E(T)$, there exists $e\in E(S)$ such that $f=\varphi (e)$.
\end{definition}
\begin{remark}
Given a strict local isomorphism along which idempotents lift, regular elements also lift by Lemma 3.1 in \cite{LwM}.
\end{remark}
Strict local isomorphisms along which idempotents lift appear in a covering theorem by Rees matrix semigroups (cf. Theorem 3.2 in \cite{LwM}). Laan and M\'{a}rki \cite{LM} also use such morphisms to describe strong Morita equivalence of semigroups with local units. 
\vspace{1em}

A semigroup $S$ is said to have \textbf{common weak local units} if for every $s,s'\in S$ there exist $u,v\in S$ such that $us=s, us'= s'$ and $sv=s, s'v=s'$. Semigroups with common weak local units are introduced in \cite{LMR} and also shown to be firm (cf. Proposition 2.4).
\begin{lemma}\label{Kristo_P3.7}
Let $\varphi :S\to T$ be a semigroup morphism. Assume that $S$ has common weak local units.
The following are equivalent.
\begin{itemize}
    \item[]1. $\varphi$ is almost injective.
    \item[]2. $\varphi \vert _{sS}$ is injective for every $s\in S$.
    \item[]3. $\varphi \vert _{Ss}$ is injective for every $s\in S$.
\end{itemize}
\end{lemma}

\begin{proof}
1. $\Rightarrow$ 2. 
Assume $\varphi$ is almost injective and take $s\in S$. Take $s',s''\in S$ such that $\varphi (ss') = \varphi (ss'')$. Since $S$ has common weak local units, there exists $u\in S$ such that $ss'=ss'u$ and $ss''=ss''u$. 
Now $\varphi (ss'u) = \varphi (ss''u)$ implies 
$ss'u = ss''u$ and hence $ss'=ss''$. The implication $1. \Rightarrow 3.$ is proved similarly. Implications $2. \Rightarrow 1.$ and $3. \Rightarrow 1.$ hold with no restrictions to $S$.
\end{proof}
\begin{remark}
A ring $R$ is called \textbf{s}-\textbf{unital} if for every $s\in R$, there exist $u,v\in R$ such that $us = s = sv$. Tominaga \cite{Tom} showed that this implies every finite non-empty subset $F\subseteq R$ admits $u,v\in R$ such that $uf = f = fv$ for every $f\in F$ (cf. Theorem 1). In particular, in present terminology, a ring has weak local units if and only if it has common weak local units. For semigroups, however, this is false. Any right zero semigroup does have local units, but does not have common weak local units. The same is true of rectangular bands.
\end{remark}

For an $R$-module $M_R$, where $R$ is a ring (not necessarily with identity) and a morphism $f:M_R\to R_R$ of $R$-modules, the set $M$ can be turned into a ring, where $f$ then becomes an almost injective morphism of rings and conversely, every strict local isomorphism $S\to R$ is, essentially, an $R$-valued linear functional \cite{Valjako}. A similar idea works in the semigroup case.
\begin{proposition}\label{Kristo_P3.9}
Let $S$ be a semigroup, $A_S$ an $S$-act and $\rho :A_S \to S_S$ an $S$-morphism. The following statements hold.
\begin{itemize}
    \item[]1. The set $A$ is a semigroup under multiplication $a\cdot a' := a\rho (a')$. The $S$-morphism $\rho$ is an almost injective semigroup morphism. If $\rho$ is also surjective, then idempotents lift along $\rho$.
    \item[]2. If $T$ is a semigroup with common weak local units, then all strict local isomorphisms $T\to S$ arise in the manner specified in 1.
\end{itemize}
\end{proposition}
\begin{proof}
For the first item we have associativity due to the equalities
\begin{align*}
    (a\cdot a')\cdot a'' = (a\rho (a'))\cdot a'' = a\rho (a')\rho (a'') = a\rho (a'\rho (a'')) = a\cdot (a'\cdot a''),
\end{align*}
where $a,a',a''\in A$. It is clear that $\rho$ is a semigroup morphism. We show $\rho$ is injective on subsemigroups of the form $a\cdot A$, where $a\in A\cdot a$. Let $a = a'\cdot a$ for some $a'\in A$. Take $m,n\in A$ such that $\rho (a\cdot m) = \rho (a\cdot n)$. Then
\[
    a\cdot m = a'\cdot (a\cdot m) = a'\rho (a\cdot m) = a'\rho (a\cdot n) = a'\cdot (a\cdot n) = a\cdot n.
\]
Assume $\rho$ is surjective and let $e=\rho (a)$ be an idempotent. Then
\[
    a^4 = a\rho (a^3) = a(\rho (a))^3 = a\rho (a) = a^2\in E(A)
\]
and $\rho (a^2) = (\rho (a))^2 = e$.

Now let $\tau :T\to S$ be a strict local isomorphism, where $T$ has common weak local units. Define $t\star s' := tt'$, where $t,t'\in T$ and $s'\in S$ such that $\tau (t') = s'$. Suppose $\tau (t'') = s'$ for some $t''\in T$. Then $\tau (tt') = \tau (t)s' = \tau (tt'')$. By Lemma \ref{Kristo_P3.7}, the map $\tau \vert _{tT}$ is injective, hence $tt
' = tt''$. Thus, the map $\star$ is well defined. Now take $s',s''\in S$, $t\in T$ and assume $\tau (t')=s'$ and $\tau (t'')=s''$ for some $t',t''\in T$. It follows that
\begin{align*}
    (t\star s') \star s'' = tt'\star s'' = tt't'' = t\star (s's''),
\end{align*}
where the last equality holds due to $s's'' = \tau (t't'')$. Thus, we have a right $S$-action on $T$. The equalities
\[
    \tau (t\star s') = \tau (tt') = \tau (t)\tau (t') = \tau (t)s'
\]
show that $\tau$ is an $S$-morphism.
\end{proof}
As we are naturally provided with two Morita semigroups, a Morita context also yields two almost injective semigroup morphisms.
\begin{theorem}\label{LM_T13_g}
Let arbitrary non-empty semigroups $S$ and $T$ be connected by a Morita context $(S,T,{}_SP_T,{}_TQ_S,\theta,\phi)$. Then $\theta$ and $\phi$ are almost injective semigroup morphisms. If $\theta$ is surjective, then idempotents  lift along $\theta$. The same holds for $\phi$.
\end{theorem}

\begin{proof}
It suffices to prove these statements for $\theta$. The map $\theta:(P\ot _TQ)_S \to S_S$ is an $S$-morphism by definition. Additionally, for every $p\ot q, p'\ot q'\in P\ot _TQ$, we have $(p\ot q)(p'\ot q') = p\ot q\theta (p'\ot q')$. Then $\theta$ is almost injective by the first item of Proposition \ref{Kristo_P3.9}. Therefore, if $\theta$ is also surjective, idempotents lift along $\theta$.
\end{proof}

The following example illustrates that almost injective morphisms need not be injective.
\begin{example}
 Let $S$ be factorizable such that it is not firm. An example of such a semigroup can be found in \cite{LMR}. Then $S\ot _S S$ is firm both as a biact and a semigroup by Theorem 2.6 in \cite{LR} and $\mu :S\otimes _S S \to S,\ s\ot s'\mapsto ss',$ is a biact morphism in the Morita context connecting $S$ and $S\ot _S S$ (cf. Proposition 4.7 in \cite{LR}). The morphism $\mu$ is a strict local isomorphism by Theorem \ref{LM_T13_g}, but it cannot be injective, because $S$ is not firm.
\end{example}

Theorem \ref{LM_T13_g} also generalizes necessity part of Theorem~13 in \cite{LM} from semigroups with local units to 
factorizable semigroups. It could also be considered as an 
analogue of Theorem~3 in \cite{LM}, where  
Rees matrix semigroups are replaced by Morita semigroups. 

\begin{corollary}
If $S$ and $T$ are strongly Morita equivalent 
semigroups then there exists a surjectively defined unitary 
Morita semigroup $Q\otimes_S P$ and a strict local 
isomorphism $\tau: Q\otimes_S P\to T$ along which 
idempotent and regular elements lift.
\end{corollary}

The following strengthens considerably Theorem 13 in
\cite{LM}. Firstly, only firmness is assumed from the semigroups 
instead of having local units and, secondly, strict local 
isomorphisms are replaced by isomorphisms. It could also be viewed as a semigroup theoretic analogue of Theorem 2.6 in \cite{Anh}.

\begin{theorem}\label{sme_firm}
Let $S$ and $T$ be firm semigroups. The following are equivalent.
\begin{itemize}
    \item[]1. $S$ and $T$ are strongly Morita equivalent.
    \item[]2. $S$ is isomorphic to a surjectively defined Morita semigroup over $T$.
\end{itemize}
\end{theorem}
\begin{proof}
For $1.\Rightarrow 2.$ assume the firm semigroups $S$ and $T$ are 
strongly Morita equivalent. By Theorem 5.9 in \cite{LMR}, they are connected by a unitary Morita 
context $(S,T,{}_SP_T,{}_TQ_S,\theta,\phi)$ with bijective mappings. Then $P\otimes _TQ$ defined by $\phi\circ \ot:Q\times P\to T$ is a unitary surjectively defined Morita semigroup over $T$. 
Similarly to Example~\ref{ex1.3}, $\theta: P\otimes_T Q\to S$ also respects the semigroup structure and is therefore an isomorphism of semigroups.\par
\noindent For $2.\Rightarrow 1.$, assume that $S$ is isomorphic to a
surjectively defined Morita semigroup $P\otimes_T Q$ over $T$. By Theorem 5 in \cite{T}, the Morita semigroup $P\otimes_T Q$ is strongly 
Morita equivalent to $T$. Using transitivity, we conclude 
that the semigroups $S$ and $T$ are strongly Morita equivalent.
\end{proof}

Hotzel \cite{Hz} noted that a surjectively defined unitary Morita semigroup over a monoid with free acts is a coordinate-free copy of a Rees matrix semigroup over that monoid. Laan and M\'{a}rki \cite{LM} showed that this is true of semigroups with weak local units (cf. Proposition 10). We can make use of their construction to show the following.

\begin{theorem}\label{LM_P10_g}
Let $S$ be a factorizable semigroup and $\mathcal M := \mathcal M(S,U,V,p)$ a Rees matrix semigroup over $S$. Then there exists a unitary Morita semigroup $Q\ot _SP$ and a strict local isomorphism $Q\ot _SP \to \mathcal M$ along which idempotents lift.
\end{theorem}
\begin{proof}
Put $Q_S := (U\times S)_S$ and ${}_SP := {}_S(S\times V)$, where the $S$-action on $Q$ is defined by $(u,s)s' := (u,ss')$ and similarly for $P$.
Due to factorisability of $S$, they are unitary $S$-acts. Define 
\[
 \sk{,} : {}_S(P \times Q)_S \to {}_SS_S,\quad \langle (s,v),(u,s')\rangle = s\,p(v,u)\,s'.
\]
One readily verifies $\sk{,}$ is an $(S,S)$-biact morphism. Define
\[
    \psi : Q\otimes _{S} P \to \mathcal M,\quad (u,s)\otimes (t,v) \mapsto (u,st,v).
\]
Consider the corresponding map $\hat{\psi} : Q\times P \to \mathcal M$. The equalities
\begin{align*}
    \hat{\psi} \left ( (u,s)s_0,(t,v) \right ) = (u,(ss_0)t,v) = (u,s(s_0t),v) = \hat{\psi} \left ( (u,s),s_0(t,v) \right )
\end{align*}
show that $\hat{\psi}$ is $S$-balanced, therefore $\psi$ is well defined by the universal property of the tensor product. Surjectivity of $\psi$ is clear.

Since $S$ is factorizable, every element in $\mathcal M$ can be written as $(u_0,s_0t_0,v_0)$ for some $u_0\in U, v_0\in V$ and $s_0,t_0\in S$. Define $\star : (Q\ot _SP) \times \mathcal M \to Q\ot _SP$ with the equality
\begin{align*}
    \left ((u,s) \ot (t,v)\right ) \star (u_0,s_0t_0,v_0) := (u,s) \ot \sk{(t,v),(u_0,s_0)}(t_0,v_0).
\end{align*}
Take $(u_0,s_0t_0,v_0)\in\mathcal M$ and denote the multiplication of the Morita semigroup $Q\ot _SP$ by the symbol $\cdot$. For every $(u,s)\ot (t,v)\in Q\ot _SP$ we have the equality
\begin{align*}
    \left ((u,s)\ot (t,v)\right ) \star (u_0,s_0t_0,v_0) = \left ((u,s)\ot (t,v)\right )\, \cdot \, \left ((u_0,s_0)\ot (t_0,v_0)\right ).
\end{align*}
Thus, the maps $- \star (u_0,s_0t_0,v_0)$ and $- \cdot \left ((u_0,s_0) \ot (t_0,v_0)\right )$ coincide on $Q\ot _SP$. It follows that $\star$ is well defined on $Q\ot _SP$. Take $(u,s)\ot (t,v)\in Q\ot _SP$. It is clear that in the event $s_0t_0 = s_0't_0'$, we have
\[
    \left ((u,s)\ot (t,v)\right ) \star (u_0,s_0t_0,v_0) = \left ((u,s)\ot (t,v)\right ) \star (u_0,s_0't_0',v_0).
\]
Take also $(u_1,s_1t_1,v_1)\in\mathcal M$. We have the equalities
\begin{align*}
    &\left (\left ((u,s)\ot (t,v)\right ) \star (u_0,s_0t_0,v_0)\right ) \star (u_1,s_1t_1,v_1)\\
    =&\, \left ((u,s) \ot (t,v)\right )\, \cdot \, \left ((u_0,s_0)\ot (t_0,v_0)\right ) \, \cdot \, \left ((u_1,s_1)\ot (t_1,v_1)\right ) \\
    =&\, \left ((u,s) \ot (t,v)\right )\, \cdot \, \left ((u_0,s_0) \ot t_0\,p(v_0,u_1)\,(s_1t_1,v_1)\right ) \\
    =&\, \left ((u,s) \ot (t,v)\right ) \star (u_0,s_0t_0\,p(v_0,u_1)s_1t_1,v_0) \\
    =&\, \left ((u,s) \ot (t,v)\right ) \star \left ( (u_0,s_0t_0,v_0)(u_1,s_1t_1,v_1) \right ).
\end{align*}
Therefore, $\star$ is an $\mathcal M$-action. We also have the equalities
\begin{align*}
    \psi  \left ( \left ((u,s)\ot (t,v)\right ) \star (u_0,s_0t_0,v_0) \right )
    &= \psi  \left ( (u,s) \ot t\,p(v,u_0)\,s_0(t_0,v_0)  \right ) \\
    &= (u,st\,p(v,u_0)\,s_0t_0,v_0) \\
    &= (u,st,v)(u_0,s_0t_0,v_0) \\
    &= \psi ((u,s)\ot (t,v))\,(u_0,s_0t_0,v_0),
\end{align*}
which implies $\psi$ is an $\mathcal M$-morphism. Due to the equality
\[
    \left ((u,s)\ot (t,v)\right ) \cdot \left ((u_0,s_0)\ot (t_0,v_0)\right ) = \left ((u,s)\ot (t,v)\right ) \star \psi ((u_0,s_0)\ot (t_0,v_0))
\]
we have by Proposition \ref{Kristo_P3.9} that $\psi$ is a strict local isomorphism along which idempotents lift.
\end{proof}

\begin{corollary}
Let $S$ be a factorizable semigroup and $\mathcal M$ a Rees matrix semigroup over $S$. Then $\mathcal M$ is a quotient of a unitary Morita semigroup.
\end{corollary}
It also turns out that the construction given by Laan and M\'{a}rki \cite{LM} yields an isomorphism if $S$ is firm.
\begin{corollary}
Let $S$ be a firm semigroup and $\mathcal M := \mathcal M(S,U,V,p)$ a Rees matrix semigroup over $S$. Then $\mathcal M$ is isomorphic to a unitary Morita semigroup over $S$. If $S=S\mathrm{im}(p)S$, then $\mathcal M$ is isomorphic to a surjectively defined unitary Morita semigroup over $S$.
\end{corollary}

\begin{proof}
Assume the construction given in the proof of Theorem \ref{LM_P10_g}. It is clear that the map $\sk{,}$ is surjective if and only if the equality $S=S\im(p)S$ holds. It remains to show that $\psi$ is injective.
\vspace{1em}

Let the equality $(u,st,v) = (u,s't',v)$ hold in $\mathcal M$. Then $st=s't'$ if and only if $s\ot t = s'\ot t'$ in $S\ot _SS$ due to firmness of $S$. Thus, we have an $S$-tossing connecting $s\ot t$ and $s'\ot t'$, which we may extend to the following $S$-tossing
\begin{equation*}
 \begin{array}{lclcrcl}
   & & & & r_1(y_1,v) & = & (t,v) \\
  (u,s)r_1 & = & (u,x_1)s_1 & & r_2(y_2,v) & = & s_1(y_1,v) \\
  (u,x_1)r_2 & = & (u,x_2)s_2 & & r_3(y_3,v) & = & s_2(y_2,v)  \\
  &\ldots & & & &\ldots & \\
  (u,x_{n-2})r_{n-1} & = & (u,x_{n-1})s_{n-1} & & r_n(y_n,v) & = & s_{n-1}(y_{n-1},v) \\
  (u,x_{n-1})r_n & = & (u,s')s_n & & (t',v) & = & s_n(y_n,v) \\
 \end{array}
\end{equation*}
where $x_i,y_i\in S$ and $r_i,s_i\in S^1$. Equivalently, the equality
\[
    (u,s)\ot (t,v) = (u,s')\ot (t',v)
\]
holds in $Q\ot _SP$. Therefore, $\psi$ is injective.
\end{proof}
\begin{remark}
For a Rees matrix semigroup $\mathcal M := \mathcal M(S,U,V,p)$ over a factorizable semigroup $S$, the condition $S=S\im(p)S$ is equivalent to $\mathcal M$ being factorizable, which, in turn, is equivalent to $S$ being strongly Morita equivalent to $\mathcal M$ by Proposition 2 in \cite{LM}.
\end{remark}

By Laan and Reimaa \cite{LR}, we have that a semigroup $S$ is factorizable if and only if $S\otimes S$ is a firm semigroup. In the same article it is shown that $A\otimes _S B = A\otimes _{S\otimes S}B$ holds for factorizable semigroups $S$. While we do not know, whether Theorem \ref{LM_P10_g} holds for factorizable semigroups, we can conclude the following.

\begin{corollary}
Let $S$ be a factorizable semigroup and $\mathcal M := \mathcal M(S\otimes S,U,V,p)$ a Rees matrix semigroup over $S\otimes S$. Then $\mathcal M$ is isomorphic to a unitary Morita semigroup over $S$, which may be assumed to be surjectively defined if $\mathcal M$ is factorizable.
\end{corollary}

\section{Dual pairs and Morita semigroups}

Hotzel \cite{Hz} considers Morita 
semigroups over monoids with $0$. His acts also must have 
a fixed zero element. We show that the mapping presented in Theorem 2.4 in \cite{Hz} is a well-behaved morphism, in general. It also turns out that it is an isomorphism for dual pairs over a semigroup with weak local units.
\vspace{1em}

We would prefer to use notation that is somewhat different from Hotzel's notation. 
For example, we write $S$ instead of $D$ and we do not assume that $S$ is a monoid with zero.

\begin{definition}
A {\bf pair} over a semigroup $S$ consists of 
\begin{itemize}
\item a left act ${_SA}$, 
\item a right act $B_S$, 
\item an $(S,S)$-biact morphism $\sk{,}: {_SA}\times B_S \to {_SS_S}$. 
\end{itemize}
Denote such a pair by $\beta$ and write $ \beta=({_SA},B_S)$.
\end{definition}
Hotzel assumes the acts in a pair are unitary, but we do not need to assume this. Every pair induces a Morita semigroup which is denoted by $B\otimes^\beta_S A$ and defined by $\sk{,}$ \footnote{Hotzel did not call these Morita semigroups. We use Talwar's terminology for the construction inspired by Hotzel's ( cf. p. 386 in \cite{T}).}. Often it is clear $B\ot _S^\beta A$ is taken with respect to $\beta$, so the superscript $\beta$ is omitted. With $\beta$ is associated the following subsemigroup of 
$\END({_SA}) \times \END(B_S)$:
\[
\Omega^\beta := \{(\rho,\sigma)\mid \rho\in \END({_SA}), \sigma\in \END(B_S), 
\sk{\rho (a),b} = \sk{a,\sigma(b)}\mbox{ for all } a\in A,b\in B\}\,.
\]
The multiplication on $\END(_SA) \times \END(B_S)$ is given by the equality
\[
    (f,g)(f',g') := (f'f,gg').
\]

Hotzel refers to such $\rho$ and $\sigma$ as linked endomorphisms. We call such $\rho$ and $\sigma$ \textbf{adjoint} endomorphisms as does \'{A}nh \cite{Anh}.

\begin{example}
Every Morita context induces a number of pairs of adjoint endomorphisms. For a given Morita context $(S,T,P,Q,\theta,\phi)$, we consider the biact morphism 
\[
\sk{\;,\;}: {_SP}\times Q_S \to S, \;\; (p,q)\mapsto \theta(p\ot q). 
\]
For fixed elements $p_0\in P$ and $q_0\in Q$
\begin{align*}
\rho & := \theta(-\ot q_0)p_0\in \END (_SP), \\ 
\sigma & := q_0\theta(p_0\ot -)\in \END (Q_S)
\end{align*}
are adjoint. Indeed, for any $p\in P$ and $q\in Q$
\begin{align*}
    \sk{\rho (p), q} = \sk{\theta (p\otimes q_0)p_0, q} &= \theta (\theta (p\otimes q_0)p_0 \otimes q) \\
    &= \theta (p\otimes q_0)\theta (p_0\otimes q) \\
    &= \theta (p\otimes q_0\theta (p_0\otimes q)) \\
    &= \sk{p,q_0\theta (p_0\otimes q)} \\
    &= \sk{p, \sigma (q)}.
\end{align*}

A symmetric construction works for $\phi$. 
\end{example}

We have an ideal of $\Omega ^\beta$:
\[
    \Omega _1^\beta := \left \{ (\rho,\sigma)\in \Omega ^\beta \mid \exists a\in A, \exists b\in B,\ \rho (A) \subseteq Sa\quad\mbox{and}\quad \sigma (B) \subseteq bS \right \}.
\]
Elements of $\Omega_1^\beta$ are sometimes called adjoint endomorphisms of rank one \footnote{In the terminology of analysis, a linear operator of rank one has a one-dimensional image. In the present case, we do not require equality.}.
We also have another ideal of $\Omega^\beta$: 
\[
\Sigma^\beta = \{(\rho,\sigma)\in \Omega^\beta\mid 
\exists b\in B,\exists a\in A,\  \rho= \sk{-,b}a \mbox{ and } \sigma= b\sk{a,-}\}\subseteq \Omega ^\beta _1.
\] 
Such pairs are denoted with the symbol $[b,a]$. So,
\[
\lbrack b,a \rbrack = \left (\sk{-,b}a, b\sk{a,-}\right ).
\]
Note that for $[b,a],[b',a']\in \Sigma ^\beta$ we have
\begin{align*}
    [b,a][b',a'] &= \left ( \sk{-,b}a\,;\,b\sk{a,-} \right )\,\left ( \sk{-,b'}a'\,;\,b'\sk{a',-}\right ) \\
    &= \left ( \sk{-,b}\sk{a,b'}a' \,;\, b\sk{\sk{a,b'}a',-} \right ) \\
    &= [b, \sk{a,b'}a'].\tag{3.1}\label{Sigma_product}
\end{align*}
Before we proceed, we will justify the above.
\begin{proposition}
The following statements hold.
\begin{itemize}
    \item[]1. The subset $\Omega ^\beta$ is a submonoid.
    \item[]2. The subset $\Omega _1^\beta$ is an ideal in $\Omega ^\beta$.
    \item[]3. The subset $\Sigma ^\beta \subseteq \Omega _1^\beta$ is an ideal in $\Omega ^\beta$.
\end{itemize}
\end{proposition}
\begin{proof}
    For the first item take $(\rho _1,\sigma _1), (\rho _2,\sigma _2)\in \Omega ^\beta$ and let $a\in A, b\in B$. Then
    \begin{align*}
        \sk{\rho _2\rho _1(a), b} = \sk{\rho _1(a), \sigma _2(b)} = \sk{a, \sigma _1\sigma _2(b)}.
    \end{align*}
    Thus, the morphisms $\rho _2\rho _1$ and $\sigma _1\sigma _2$ are adjoint. As $\ID_A$ and $\ID_B$ are clearly adjoint, $\Omega^\beta$ is a monoid.
    
    For the second item take $(\rho,\sigma)\in \Omega ^\beta$ and $(\rho_1,\sigma_1)\in \Omega _1^\beta$. There exist $a_1\in A$ and $b_1\in B$ such that $\rho _1(A) \subseteq Sa_1$ and $\sigma _1(B) \subseteq b_1S$. Then $\rho _1\rho (A)\subseteq \rho _1(A) \subseteq Sa_1$ and
    $
        \sigma\sigma _1(B) \subseteq \sigma (b_1S) = \sigma (b_1) S.
    $
    Thus $\Omega _1^\beta$ is a left ideal in $\Omega ^\beta$. The right ideal case is proved similarly.
    
    For the third item let $(\rho_1,\sigma_1)\in\Sigma ^\beta$, i.e, $\rho _1 = \sk{-,b_1}a_1$ and $\sigma _1 = b_1\sk{a_1,-}$ for some $a_1\in A$ and $b_1\in B$. The inclusion $(\rho_1,\sigma _1) \in \Omega _1^\beta$ is clear. Take $(\rho,\sigma) \in \Omega ^\beta$ and note that for every $x\in A$ we have
    \begin{align*}
        \rho _1\rho (x) = \sk{\rho (x),b_1}a_1 = \sk{x,\sigma (b_1)}a_1
    \end{align*}
    and, on the other hand, for every $y\in B$ we have
    \begin{align*}
        \sigma \sigma _1 (y) = \sigma \left (b_1\sk{a_1,y} \right ) = \sigma (b_1) \sk{a_1,y}.
    \end{align*}
    Therefore, $\rho _1\rho = \sk{-,\sigma (b_1)}a_1$, $\sigma\sigma _1 = \sigma (b_1)\sk{a_1,-}$ and $\Sigma ^\beta$ is a left ideal in $\Omega ^\beta$. The right ideal case is proved similarly.
\end{proof}

\begin{definition}
A pair $\beta = ({_SA},B_S)$ is called {\bf dual} if 
\begin{itemize}
\item[(1)] $\forall a\in A,\exists a'\in A,\quad a\in Sa' \mbox{ and }\sk{a',B}=S$,  
\item[(2)]  $\forall b\in B,\exists b'\in B,\quad b\in b'S \mbox{ and }\sk{A,b'}=S$.
\end{itemize}
\end{definition}

Examples of dual pairs can be found in Hotzel \cite{Hz}. Note that acts ${_SA}$ and $B_S$ in a dual pair are necessarily unitary. By Theorem 2.5 in \cite{Hz}, the equality $\Sigma ^\beta = \Omega _1^\beta$ holds for a dual pair over monoid with $0$. This is also true in case of a dual pair over a semigroup with weak local units.
\begin{theorem}[cf. Theorem 2.5 in \cite{Hz}]\label{Hz_2.5}
Let $\beta = (_SA,B_S)$ be a dual pair over a semigroup $S$ with weak local units. Then $\Sigma ^\beta = \Omega _1^\beta$.
\end{theorem}

\begin{proof}
The inclusion $\Sigma^\beta \subseteq \Omega_1^\beta$ is clear.  
We show $\Omega _1^\beta \subseteq \Sigma ^\beta$. Take $(\rho,\sigma)\in \Omega _1^\beta$, that is, $\rho (A) \subseteq Sa_1$ and $\sigma (B) \subseteq b'S$ for some $a_1\in A$ and $b'\in B$. We must show that there exist $a\in A$ and $b\in B$ such that $\rho = \sk{-,b}a$ and $\sigma = b\sk{a,-}$.
\vspace{1em}

Since we have a dual pair, $a_1 = sa_2$ for some $s\in S$ and $a_2\in A$. Due to presence of weak local units in $S$ take $s_r\in S$ such that $s=ss_r$ and let $\sk{a_2,b_2} = s_r$ for some $b_2\in B$. Similarly, $b' = b''t$, $t=t_\ell t$ for some $t_\ell \in S$ and $\sk{a'',b''} = t_\ell$ for some $a''\in A$. 
Then $\sigma (b_2) = b'v'$ and $\rho (a'') = u'a_1$ for some $u',v'\in S$. 
Putting $v:= tv'$ and $u:=u's$ we have 
\[
\sigma(b_2) = b'v' = b''tv' = b''v \;\; \mbox{ and } \;\; \rho(a'') = u'a_1 = u'sa_2 = ua_2. 
\]
Hence 
\begin{align*}
u & = u's = u'ss_r = u's\sk{a_2,b_2} = \sk{u'sa_2,b_2} = \sk{\rho(a''),b_2} \\ 
& = \sk{a'', \sigma(b_2)} = \sk{a'',b''tv'} = \sk{a'',b''}tv' = t_\ell tv' = tv' = v.
\end{align*}
Take $x\in A$, then $\rho (x) = za_1$ for some $z\in S$ and
\[
    \sk{\rho (x), b_2}a_2 = \sk{zsa_2,b_2}a_2 = zs\sk{a_2,b_2}a_2 = zss_ra_2 = zsa_2 = za_1 = \rho (x)
\]
and therefore, for every $x\in A$ we have
\[
    \rho (x) = \sk{\rho (x), b_2}a_2 = \sk{x,\sigma (b_2)}a_2 = \sk{x,b''}va_2.
\]
Similarly, for a fixed $y\in B$, we have $\sigma (y) = b'w$ for some $w\in S$ and
\[
    b''\sk{a'',b'w} = b''\sk{a'', b''tw} = b''\sk{a'',b''}tw = b''t_\ell tw = b''tw = b'w = \sigma (y).
\]
Therefore, for every $y\in B$ we have
\[
    \sigma (y) = b''\sk{a'',\sigma (y)} = b'' \sk{\rho (a''),y} = b''\sk{ua_2,y} = b''\sk{va_2,y}.
\]
Thus, it suffices to take $b=b''$ and $a=ua_2$.
\end{proof}

The following is a semigroup theoretic analogue for Proposition 2.2 in \cite{Anh}.
\begin{lemma}\label{Anh_2.2_sgp}
Let $\beta = ({}_SA,B_S)$ be a dual pair, where $S$ is a semigroup with weak local units. Then the Morita semigroup $B\ot _S^\beta A$ has weak local units. If $S$ has local units, then $B\ot _S^\beta A$ also has local units.
\end{lemma}
\begin{proof}
Take $b\ot a\in B\ot _S A$. Since we have a dual pair, there exist $a_1\in A$ and $s\in S$ such that $a=sa_1$. Since $S$ has weak local units, we can write $s=su$ for some $u\in S$. By duality, we also have $u=\sk{a_1,b_1}$ for some $b_1\in B$. Then 
\[
a= sa_1 = sua_1 = s \sk{a_1,b_1}a_1 = \sk{sa_1,b_1} a_1 = \sk{a,b_1} a_1, 
\]
hence 
\[
b\ot a = b\ot \sk{a,b_1} a_1 = (b\ot a)(b_1\ot a_1). 
\]
If $S$ has local units, then we may assume $u$ is an idempotent and we have that
\[
    (b_1\ot a_1)^4 = b_1 \ot \sk{a_1,b_1}^3a_1 = b_1 \ot \sk{a_1,b_1}a_1 = (b_1\ot a_1)^2 \in E \left ( B\ot_S A \right ).
\]
Similarly, $b=b_2t$ and $t=vt$, where $v=\sk{a_2,b_2}$ for some $a_2\in A$. The equality $b\ot a = (b_2\ot a_2)(b\ot a)$ follows.
\end{proof}
\vspace{1em}

It turns out that Hotzel's construction (cf. Theorem 2.4 in \cite{Hz}) yields a morphism with good properties in case of an arbitrary pair.
\begin{theorem}\label{Hz_sli}
For any pair $\beta= ({_SA},B_S)$ over an arbitrary non-empty semigroup $S$ there exists a strict local isomorphism from $B\otimes ^\beta _SA$ onto $\Sigma ^\beta$ along which idempotents lift.
\end{theorem}
\begin{proof}
Define the same way as in \cite{Hz}
\[
    \varphi : B\otimes _S  A \to \Sigma ^\beta,\quad b\otimes a \mapsto [b,a].
\]
The corresponding map 
\[
B\times A\to \Sigma ^\beta,\quad (b,a) \mapsto [b,a],
\]
is $S$-balanced. Indeed, for any $s\in S$ we have the equalities
\[
    \sk{-,bs}a = \sk{-,b}sa \quad\mathrm{and}\quad bs\sk{a,-} = b\sk{sa,-}.
\]
Thus, $\varphi$ is well defined. Surjectivity is clear. Define $\star : (B\ot _S A) \times \Sigma ^\beta \to B\ot _SA$ with the equality
\[
    (b\ot a)\star [b',a'] := b\ot \sk{a,b'}a'.
\]
Let $b'\in B$ and $a'\in A$ be fixed, then for every $b\ot a\in B\ot _SA$ we have the equality $(b\ot a) \star [b',a'] = (b\ot a)(b'\ot a')$. Thus, the maps $- \star [b',a']$ and $- \cdot b'\ot a'$ coincide on $B\ot _SA$. It follows that $\star$ is well defined. Take $b\ot a, b''\ot a''\in B\ot _SA$. The equalities
\begin{align*}
    \left ( b\ot a \star [b'\ot a']\right ) \star b''\ot a'' &= (b\ot a)(b'\ot a')(b''\ot a'') \\
    &= (b\ot a) \left (b'\ot \sk{a',b''}a''\right ) \\
    &= (b\ot a) \star \left [b', \sk{a',b''}a''\right ] \\
    &= (b\ot a) \star \left ( [b',a']\,[b'',a''] \right ) \tag{cf. \ref{Sigma_product}}
\end{align*}
show that $\star$ is a $\Sigma ^\beta$-action. We also have
\begin{align*}
    \varphi \left ( (b\ot a) \star [b',a']  \right ) &= \varphi \left ( b\ot \sk{a,b'}a' \right ) \\
    &= \left [b,\sk{a,b'}a'\right ] \\
    &= [b,a]\,[b',a'] \\
    &= \varphi (b\ot a)\, \left [b',a'\right ].
\end{align*}
Thus, $\varphi$ is a $\Sigma ^\beta$-morphism. Due to
\[
    (b\ot a)(b'\ot a') = (b\ot a) \star \varphi (b'\ot a')
\]
we have by Proposition \ref{Kristo_P3.9} that $\varphi$ is a strict local isomorphism along which idempotents lift. 
\end{proof}

By Theorem 2.4 in \cite{Hz}, for a dual pair $\beta$ over a monoid with zero, the Morita semigroup $B\otimes ^\beta _SA$ is isomorphic to the semigroup $\Sigma ^\beta$. This is also true for semigroups with weak local units.

\begin{theorem}\label{Hz_2.4}
Let $S$ be a semigroup with weak local units and $\beta = (_SA,B_S)$ a dual pair. Then the Morita semigroup $B\otimes ^\beta_SA$ is isomorphic to the semigroup $\Sigma ^\beta$.
\end{theorem}

\begin{proof}
It suffices to show that the map $\varphi$ from the proof of Theorem \ref{Hz_sli} is injective. Suppose
$
    [b,a] = [b',a']
$
for some $b\otimes a, b'\otimes a'$ in $B\otimes _S A$. Said equality means that
\begin{align*}
&\forall a''\in A,\quad \sk{a'',b}a = \sk{a'',b'}a', \\ 
 &\forall b''\in B,\quad  b\sk{a,b''} = b'\sk{a',b''}. 
\end{align*}
By Lemma \ref{Anh_2.2_sgp}, take $b_1\ot a_1, b_2\ot a_2\in B\ot _S A$ such that
\[
    b\ot a = (b\ot a)(b_1\ot a_1) \quad\mathrm{and}\quad b'\ot a' = (b_2\ot a_2)(b'\ot a').
\]

We then have
\begin{align*}
b\ot a & = (b\ot a)(b_1\ot a_1) \\ 
& = b\ot \sk{a,b_1}a_1 \\ 
& = b\sk{a,b_1}\ot a_1 \\ 
& = b'\sk{a',b_1}\ot a_1 \\ 
& = (b'\ot a')(b_1\ot a_1) 
\end{align*}
and therefore, 
\begin{align*}
b\ot a & = (b_2\ot a_2)(b'\ot a')(b_1\ot a_1)\\ 
& = (b_2\ot a_2)(b\ot a) \\ 
& = b_2\ot \sk{a_2,b}a \\ 
& = b_2\ot \sk{a_2,b'}a' \\
& = (b_2\ot a_2)(b'\ot a') \\ 
& = b'\ot a'.
\end{align*}
\end{proof}
We can deduce a sufficient condition for strong Morita equivalence. 

\begin{corollary}
Let $S$ and $T$ be semigroups with weak local units. If $T\cong \Sigma^{\beta}$ for some dual pair $\beta = ({_SA},B_S)$ then $S$ and $T$ are strongly Morita equivalent. 
\end{corollary}
\begin{proof}
By Theorem \ref{Hz_2.4}, $T\cong \Sigma^\beta \cong B\ot_S^\beta A$, where $B\ot_S^\beta A$ is surjectively defined. Hence $T$ and $S$ are strongly Morita equivalent by Theorem \ref{sme_firm}. 
\end{proof}

In particular, each dual pair $\beta$ over $S$ with weak local units gives rise to a semigroup $\Sigma^\beta$ which is strongly Morita equivalent to $S$.

\end{document}